\documentclass[a4paper, 11pt]{article}
\usepackage{arxiv}
\usepackage{times,cite, caption}
\usepackage[T1]{fontenc}
\usepackage[utf8]{inputenc}

\usepackage{graphicx,amsmath,amssymb,todonotes,tikz,pgfplots,mathtools,hyperref, amsthm}
\newtheorem{definition}{Definition}
\newtheorem{lemma}{Lemma}
\newtheorem{corollary}{Corollary}
\newtheorem{remark}{Remark}
\pgfplotsset{
        compat=1.14,
    }
\usetikzlibrary{
      pgfplots.colorbrewer,
}
\pgfplotsset{
    cycle list/.define={my marks}{
      every mark/.append style={solid,fill=\pgfkeysvalueof{/pgfplots/mark list fill}},mark=none\\
      every mark/.append style={solid,fill=\pgfkeysvalueof{/pgfplots/mark list fill}},mark=none\\
      every mark/.append style={solid,fill=\pgfkeysvalueof{/pgfplots/mark list fill}},mark=none\\
      every mark/.append style={solid,fill=\pgfkeysvalueof{/pgfplots/mark list fill}},mark=none\\
      every mark/.append style={solid,fill=\pgfkeysvalueof{/pgfplots/mark list fill}},mark=*\\
        every mark/.append style={solid,fill=\pgfkeysvalueof{/pgfplots/mark list fill}},mark=square*\\
        every mark/.append style={solid,fill=\pgfkeysvalueof{/pgfplots/mark list fill}},mark=triangle*\\
        every mark/.append style={solid,fill=\pgfkeysvalueof{/pgfplots/mark list fill}},mark=diamond*\\
        every mark/.append style={solid,fill=\pgfkeysvalueof{/pgfplots/mark list fill}},mark=star\\
        every mark/.append style={solid,fill=\pgfkeysvalueof{/pgfplots/mark list fill}},mark=Mercedes star\\
        every mark/.append style={solid,fill=\pgfkeysvalueof{/pgfplots/mark list fill}},mark=Mercedes star flipped\\
    }
}
\DeclarePairedDelimiter\norm{\lVert}{\rVert}
\newcommand{\NN}{{\cal N}}
\newcommand{\patch}{{\cal P}}
\hypersetup{
pdftitle={A template for the arxiv style},
pdfsubject={q-bio.NC, q-bio.QM},
pdfauthor={David S.~Hippocampus, Elias D.~Striatum},
pdfkeywords={First keyword, Second keyword, More},
}
\begin{document}

\title{A hybrid finite element/neural network solver and its application to the Poisson problem}
\author{
  Uladzislau Kapustsin 
  \qquad Utku Kaya 
  \qquad Thomas Richter \\
  Otto-von-Guericke Universit\"at Magdeburg, Germany}

\begin{abstract}
  We analyze a hybrid method that enriches coarse grid finite element solutions with fine scale fluctuations obtained from a neural network. 
  The idea stems from the \emph{Deep Neural Network Multigrid Solver} (DNN-MG) ~\cite{Margenberg2021} which embeds a neural network into a multigrid hierarchy by solving coarse grid levels directly
  and predicting the corrections on fine grid levels locally (e.g. on small patches that consist of several cells) by a neural network. Such local designs are quite appealing, as they allow a very good generalizability.  In this work, we formalize the method and describe main components of the a-priori error analysis.
  Moreover, we numerically investigate how the size of training set affects the solution quality.
\end{abstract}

\maketitle                    

\section{Introduction}
Recent advancements in employing neural networks to approximate solutions to partial differential equations (PDEs) mostly focus on
Physics Inspired Neural Networks (PINNs)~\cite{raissi2019} such as the Deep Ritz method~\cite{weinan2018}. They leverage the expressive power of neural networks while incorporating physical principles 
and promise substantial efficiency increase for high dimensional or parameter dependent partial differential equations. One main drawback of PINNs is 
that they need to re-train when the problem parameters change.
Also, for classical problems, such as three 
dimensional fluid dynamics problems, highly sophisticated and well established discretization methods regarding the efficiency and accuracy are available that beat neural network approaches by far.

The method of this paper was introduced as  main component the of DNN-MG~\cite{Margenberg2021} for the instationary Navier-Stokes equations. At each time step, a coarse solution is obtained 
by a classical finite element solver and corrections to  finer grids are predicted locally via neural networks.
Here, we focus on a simpler linear problem and aim to understand the mechanism of such a hybrid approaches by discussing its a-priori errors and via numerical experiments.

Let $\Omega\subset\mathbb{R}^d, \; d \in \{2,3\}$ be a domain with polygonal boundary. 
We are interested in the weak solution of the Poisson's equation
\begin{equation}
  -\Delta u = f, \quad 
  u \vert_{\partial \Omega} = 0,
  \label{eq:poisson equation}
\end{equation}
with a given force term $f\in H^{-1}(\Omega)$.

For a subdomain $\omega \subseteq \Omega$, let $\mathcal{T}_h(\omega)=\{T_i\}_{i=1}^M$ be a non-overlapping admissible decomposition of $\omega$ into convex polyhedral elements $T_i$ such that 
$\overline \omega = \cup_{i=1}^M \overline{T}_i$. The diameter of element $T$ is denoted by $h_T$ and $h = \max_{T \in \mathcal{T}_h(\Omega)} h_T$.
With $\|\cdot\|_2$ we denote the Euclidean norm and for $v \in C(\overline{\omega})$ we define 
\begin{equation*}
  \| v \|_{l^2(\omega)} :=
   \Big ( \sum \limits_{x \text{ is node of } \mathcal{T}_h(\omega)} v(x) \Big)^{\frac1{2}}. 
\end{equation*}
 Moreover, let $V_h^{(r)}$ be the space of piecewise polynomials of degree $r \ge 1$ satisfying the homogeneous Dirichlet condition on the boundary $\partial\Omega$, i.e.
\begin{equation*}
  V_h
  :=
  \left\{
  \phi \in C(\overline \Omega) \text{ s.t. } \; 
  \phi |_{T} \in P^{(r)}(T) \; \forall T \in \Omega_h, \;
  \phi |_{\partial \Omega} = 0
  \right\},
\end{equation*}
 where $P^{(r)}(T)$ is the space of polynomials of degree $r$ on a cell $T\in \mathcal{T}_h$. We assume that there is a hierarchy of meshes
 \begin{equation*}
  \mathcal{T}_H(\Omega) 
  := 
  \mathcal{T}_{0} 
  \preccurlyeq \mathcal{T}_{1}
  \preccurlyeq \cdots  
  \preccurlyeq \mathcal{T}_{L} 
  =: 
  \mathcal{T}_h(\Omega),
\end{equation*}
where we denote by $\mathcal{T}_{l-1} \preccurlyeq \mathcal{T}_l$, that each element of the fine mesh $T \in \mathcal{T}_{l}$ originates from the uniform refinement of a coarse element $T'\in \mathcal{T}_{l-1}$, for instance, uniform splitting of a quadrilateral or triangular element into four and of a hexahedral or tetrahedral element into eight smaller ones, respectively. 
Accordingly we have the nesting $V_h^{(l-1)} \subset V_h^{(l)},\quad l = 1, \dots, L$ where $V_h^{(l)}$ is the space defined on the mesh level $l$.
With a patch $\mathcal{P} \in \mathcal{T}_h(\Omega)$ we refer to a polyhedral subdomain of $\Omega$, but for simplicity we assume that each patch corresponds to a cell of $\mathcal{T}_H(\Omega)$.
By $V_h(\mathcal{P})$ we denote the local finite element subspace
$$
  V_h(\mathcal{P}) := \operatorname{span} \left\{
\phi_h |_{\mathcal{P}}, \; \phi_h \in V_h
\right\}
$$
and $R_{\mathcal{P}} : V_h \to V_{\mathcal{P}}$ denotes the restriction to the local patch space, defined via
\[
R_{\mathcal{P}}(u_h)(x_i) = u_h(x_i) \quad \text{for each node } x_i\in \mathcal{T}_h(\mathcal{P}).
\]
The prolongation $P_{\mathcal{P}} : V_h(\mathcal{P}) \to V_h$ is defined by
\begin{equation}\label{prolong}
P_{\mathcal{P}}(v)(x) =
\begin{cases}
  \frac{1}{n} v(x) & x \text{ is a node of } \mathcal{T}_h(\mathcal{P}),
  \qquad n\in\mathbb{N}\text{ being the number of patches containing the node $x$}
  \\
0  & \text{otherwise}.
\end{cases}
\end{equation}
The classical continuous Galerkin finite element solution of the problem \eqref{eq:poisson equation} is $u_h \in V_h$ s.t.
\begin{equation}\label{fem}
  (\nabla u_h, \nabla\phi) 
  = 
  (f, \phi) \quad \forall \phi \in V_h,
\end{equation}
with the $L^2$ inner product $(\cdot, \cdot)$. We are interested in the scenario where one prefers not to solve \eqref{fem} on the finest level $V_h$ due to lacking hardware resources or too long computational times, but in $V_H$ with $H\gg h$. 
This is the so-called coarse solution $u_H\in V_H$ and fulfills $  (\nabla u_H, \nabla\phi) 
= (f, \phi) \quad \forall \phi \in V_H$. The key idea of our method is to obtain the fine mesh fluctuations $u_h-u_H$ in forms of neural network updates $w_{\mathcal{N}}$ corresponding to the inputs $u_H$ and $f$. Hence, the neural network updated solution has the form $u_{\mathcal{N}}:=u_H + w_{\mathcal{N}}$ in the case where the network operates globally on the whole domain.
A more appealing setting is where these updates are obtained locally, such that the network is acting on the data not on the whole domain at once, but on small patches $\mathcal{P} \in \mathcal{T}_h(\Omega)$. 
In this case, while the training is performed in a global manner, the updates are patch-wise and the network updated solution has the form $u_{\mathcal{N}}:=u_H + \sum_{\mathcal{P}}P_{\mathcal{P}}w^{\mathcal{P}}_{\mathcal{N}}$.
\section{Hybrid finite element neural network discretization}

\subsection{Neural network}\label{sec:network}
In this section we introduce the neural network we use and formalize the definition of finite element/neural network solution.
\begin{definition}[\emph{Multilayer perceptron}]\label{def:MLP}
  Let $L \in \mathbb{N}$ be the number of layers and let $N_i$ be the number of neurons on layer $i \in \{1,\dots,L\}$. Each layer $i \in \{ 1, \dots ,{L-1}\}$ is associated with a nonlinear function $l_i(x) : \mathbb{R}^{N_{i-1}} \rightarrow \mathbb{R}^{N_i}$ with
  \begin{equation}\label{def:layer}
    l_i(x) = \sigma(W_i x + b_i) 
  \end{equation}
  and an \emph{activation function}
  $\sigma : \mathbb{R} \to \mathbb{R}$.
  The \emph{multilayer perceptron} (MLP) ${\cal N} : \mathbb{R}^{N_{0}} \to \mathbb{R}^{N_{L}}$ is defined via 
  \begin{equation*}
    \mathcal{N} = W_n (l_{n-1} \circ \dots \circ l_1)(x) + b_n
  \end{equation*}
  where $W_i \in \mathbb{R}^{N_{i-1} \times N_{i}}$ denote the \emph{weights}, 
  and $b_i \in \mathbb{R}^{N_i}$ the \emph{biases}.  
\end{definition}

\subsection{Hybrid solution}
On a patch $\mathcal{P}$, the network receives a tuple $(R_{\mathcal{P}}u_H, R_{\mathcal{P}}f)$, restrictions of the coarse solution $R_{\mathcal{P}}u_H$  and of the 
source term  $R_{\mathcal{P}}f$ and it returns an approximation to the fine-scale update  $v_h^{\mathcal{P}}(u_h-u_H)|_{\mathcal{P}} \in V_h(\mathcal{P})$.  In order to obtain a globally continuous 
function, the prolongation \eqref{prolong} is employed. 
\begin{definition}[\emph{Hybrid solution}]
The hybrid solution is defined as
\begin{equation}\label{def:hybrid}
    u_{\mathcal{N}} 
    := 
    u_H 
    + 
    \sum\limits_{{\mathcal{P}}} 
    P_{\mathcal{P}} w_N^{\mathcal{P}},
  \end{equation}
    where 
    $w_N^{\mathcal{P}} =\sum_{i=1}^{N} W^P_i \phi_i,$
    $W^P_i$ is the $i-$th output of $\mathcal{N}(y)$ and $\phi_i$ are the
    basis functions of $V_h(\mathcal{P})$. Here, $y = \left(
      U_{H}^{\mathcal{P}}, F_h^{\mathcal{P}}
      \right)^T$ is the input vector where $U_{H}^{\mathcal{P}}$ and $F_h^{\mathcal{P}}$ are the nodal values of $u_H$ on the coarse mesh $\mathcal{T}_H(\Omega)$ and $f$ on the mesh $\mathcal{T}_h(\mathcal{P})$, respectively.
    \end{definition}
For simplicity we will mostly use the notation 
\begin{equation*}
  u_{\mathcal{N}}    =    u_H    +    \mathcal{N}(f)
\end{equation*}
in place of \eqref{def:hybrid}.

 Since each function $u_H \in V_H$ also belongs to $V_h$, it has the form $u_H = \sum\limits_{i=1}^{N_{dof}} U^i_{Hh} \phi^i_h$ with $\{\phi^i_h\}_{i=1}^{N_{dof}}$ being the basis of the fine finite element space $V_h$ and $U_{Hh}$ being the coefficient vector of interpolation of $u_H$ into $V_h$. As we update the coarse solution $u_H$ on fine mesh nodes, this procedure can be considered as a simple update 
of coefficients $U^i_{Hh}$, i.e.
$$u_\mathcal{N} = \sum\limits_{i=1}^{N_{dof}}	(U^i_{Hh} + W^i_{\mathcal{N}}) \phi^i_h \in V_h,$$
or simply $U_{\mathcal{N}} := U_{Hh}+ W_{\mathcal{N}}$ being the coefficient vector of $u_{\mathcal{N}}$.

\subsection{Training}
The neural network is trained using fine finite element solutions obtained on the mesh $\mathcal{T}_h(\Omega)$ and with the loss function
\begin{equation}\label{loss}
  \mathcal{L}(u_h,u_H;w_h) := \frac{1}{N_T N_P}\sum_{i=1}^{N_T} \sum_{\mathcal{P}\in \Omega_h} \| (u^{f_i}_h-u^{f_i}_H)-w_\NN^{f_i} \|^2_{l^2(\mathcal{P})}
\end{equation}
where $N_T$ is the size of training set and $N_P$ is the number of patches. Here, $w_\NN^{f_i}$ stands for the finite element function defined by the network update $\mathcal{N}(f_i)$ on the patch $\mathcal{P}$. The training set ${\cal F}=\{f_1,\dots,f_{N_{tr}}\}$ consists of $N_{tr}\in\mathbb{N}$ source terms $f_i$ together with corresponding coarse and fine mesh finite element solutions $u_H^{f_i}$ and $u_h^{f_i}$, respectively. 

\section{On the a-priori error analysis}
The difference between the exact solution $u \in H^1_0(\Omega)$ of \eqref{eq:poisson equation} and the hybrid solution $u_{\mathcal{N}}$ from \eqref{def:hybrid} can be split as
\begin{equation}\label{error}
  \|u-u_{\mathcal{N}}\| \le \min_{f_i\in {\cal F}}
  \Big(
  \|u-u_{h}\| +
  \|u_{h}-u_{h}^{f_i}\| +
  \|u_{h}^{f_i}-u_{\mathcal{N}}^{f_i}\|+
  \| u_{\mathcal{N}}^{f_i}-u_{\mathcal{N}}\|\Big\},
\end{equation}
$u_{h}^{f_i}, u_{\mathcal{N}}^{f_i} \in V_h$
being the finite element solution and the neural network updated solution
corresponding to the source term $f_i$, respectively. Let us discuss individual terms
in \eqref{error}.
\begin{itemize}
  \item $u-u_h$ is the fine mesh finite element error. Estimates of this error are well-known in the literature and are of $\mathcal{O}(h^r)$ in the $H^1$ semi-norm.
  \item $ (u_{h}-u_{h}^{f_i})$ is a \emph{data approximation error} and in the $H^1$ semi-norm it can be bounded by $\|f-f_i\|_{-1}$ due to through stability of the finite element method.
  \item $(u_{h}^{f_i}-u_{\mathcal{N}}^{f_i})$ is a \emph{network approximation error} and is introduced by the approximation properties of the network architecture. This is bounded by the 
  tolerance $\epsilon$ which depends on the accuracy the minimization problem \eqref{loss}.
  \item $u_{\mathcal{N}}^{f_i}-u_{\mathcal{N}} = (u_H - u_H^{f_i}) + (\mathcal{N}(f) - \mathcal{N}(f_i))$ consists of a \emph{generalization error} of the network and a further error term depending on the richness of the data set. While the term $u_{H}^{f_i}-u_{H}$
  can be handled via the stability of the finite element method, the remaining term requires a stability estimate of the neural network.
\end{itemize}  
Overall, an estimate of
\begin{equation} \label{estimate}
  \|\nabla( u-u_{\mathcal{N}})\| \leq c \Big ( h^r\| f \|_{r+1} + \epsilon + \min_{f_i \in {\cal F}} \big\{\| f-f_i \|_{-1} + \|\nabla (\mathcal{N}(f) - \mathcal{N}(f_i))\|\big\} \Big)
\end{equation}
can be obtained for sufficiently smooth source term $f$ and domain $\Omega.$ Improvements of this estimate with the consideration of patch-wise updates is part of an ongoing work.

\subsection{Stability of the neural network}
\label{sec:stability}
The network dependent term of $\eqref{estimate}$ is linked with the stability of the network. For a study of the importance of Lipschitz regularity in the generalization bounds we refer to \cite{bartlett2017spectrally}.
\begin{lemma}\label{lemma:1}
  Let $\NN$ be a multilayer perceptron (Def. \ref{def:MLP}) and $\sigma: \mathbb{R} \to \mathbb{R}$ satisfy $ |\sigma(y)- \sigma(y_i)| \leq c_0 |y- y_i|$ with $c_0>0$. Then, on each patch $\mathcal{P}$ for the inputs $y$ and $y_i$ and the corresponding
  FE functions $\mathcal{N}(f)$ and $\mathcal{N}(f_i)$ (uniquely defined by the network updates) holds
  \begin{equation}\label{lemma:1:bound}
    \| \mathcal{N}(f)- \mathcal{N}(f_i)\|_{\mathcal{P}} \leq c\cdot c_0^{N_L} \cdot c_W \cdot  h^{d}\|y-  y^{f_i}\|_2
  \end{equation}
   where 
   \begin{equation*}
   c_W :=  \prod_{j=1}^{N_L} \|{W^j}\|_2.
   \end{equation*}
\end{lemma}

\begin{proof}
The definition of the network gives
  \begin{equation}
    \begin{split} \label{eq:network stability 1}
  \|\mathcal{N}(f) - \mathcal{N}(f_i)\|_{l^2(\mathcal{P})}     & = 
 \|W^{N_L} (z_{N_L-1}(y) - z_{N_L-1}(y^{f_i}))\|_2 
 \leq
 \|{W^{N_L}}\|_2 \cdot \|z_{N_L-1}(y) - z_{N_L-1}(y^{f_i})\|_2
    \end{split}
  \end{equation}
  where $z_i = l_i \circ \cdots \circ l_1$ and $l_i$ are as defined in \eqref{def:layer}.
  By using the definition of $z_j$ and the Lipschitz constant of $\sigma(\cdot)$ we obtain for an arbitrary layer $j$
  \begin{equation} \label{eq:layer step}
    \begin{split}
      \|{z_{j}(y) - z_{j}(y^{f_i})}\|_2
      & = 
      \|{
	\sigma(W^j z_{j-1}(y))
	-
	\sigma(W^j z_{j-1}(y^{f_i}))
      }\|_2
      \leq  c_0
      \|{
	W^j \left(
	z_{j-1}(y)
	-
	z_{j-1}(y^{f_i})
	\right)
      }\|_2
      \\
      & \leq
      c_0 \|{W^j}\|_2 
      \cdot 
      \|{z_{j-1}(y) - z_{j-1}(y^{f_i})}\|_2.
    \end{split}
  \end{equation}
  Then, by applying \eqref{eq:layer step} recursively from the second to the last layer we obtain
  \begin{equation*}
    \|{z_{N_L-1}(y) - z_{N_L-1}(y^{f_i})}\|_2
    \leq 
    c_0^{N_L-1}\prod\limits_{i=1}^{N_L-1} 
    \|{W_j}\|_2 \cdot \|{y - y^{f_i}}\|_2
  \end{equation*}
  Hence, by applying it to \eqref{eq:network stability 1} and using the inequality
  \begin{align*}
    \|v\|_{\mathcal{P}}^2 \; \leq \; & c  h^{2d}\|{v}\|_{l^2(\mathcal{P})}^2 \;  \quad \forall v \in V_h(\mathcal{P})
     \end{align*}
      we arrive at the claim.
\end{proof}
\begin{corollary}
Lemma \ref{lemma:1} leads to  
$$\| \nabla(\mathcal{N}(f) - \mathcal{N}(f_i))\| \leq c_{inv}c_1 \Big( c_{\Omega} h^{-1}\| f-{f_i} \|_{-1} + h^d \sum_{\mathcal{P}}\|f-f_i \|_{l^2(\mathcal{P})}\Big)$$
with the constant $c_1 = c\cdot c_0^{N_L} \cdot c_W$ arising from Lemma above and $c_{inv}$ and $c_{\Omega}$ arising from inverse and Poincar\'e estimates, respectively.
\end{corollary}
\begin{proof}
The definition of inputs together with the triangle inequality and the inequality 
\begin{align*}
\|{v}\|_{l^2(\mathcal{P})}^2 \; \leq \;   h^{-2d} \|v\|_{\mathcal{P}}^2  \quad \forall v \in V_h(\mathcal{P})
  \end{align*}
  provides
   $$\| y - y^{f_i}\|_2 \leq \| u_H-u^{f_i}_H \|_{l^2(\mathcal{P})}+ \|f-f_i \|_{l^2(\mathcal{P})} \leq h^{-d} \| u_H-u^{f_i}_H \|_{\mathcal{P}} + \|f-f_i \|_{l^2(\mathcal{P})}$$
  for each patch $\patch$. In the whole domain this, with Poincare\'s inequality, leads to 
  \begin{align*}
    \| \mathcal{N}(f)- \mathcal{N}(f_i)\|     \leq   c_1 \big(c_{\Omega}h^{-1} \| \nabla (u_H - u_H^{f_i})\| +h^d \sum_{\mathcal{P}}\|f-f_i \|_{l^2(\mathcal{P})} \big).
  \end{align*}
The stability of the coarse discrete solution and the inverse estimate shows the claim.
\end{proof}
\begin{remark}
  A different network architecture may include several layers that perform convolutions. This kind of networks are called convolutional neural networks.
  In the two-dimensional setting, this would correspond to replacing $l_i$ of Definition \ref{def:MLP} with a nonlinear function
  $ l^c_i : \mathbb{R}^{N_i^c \times N_i^c} \rightarrow \mathbb{R}^{N^c_{i+1} \times N^c_{i+1}}$ defined as
  $$l^c_i(x) = \sigma(W_i \ast x + b_i)$$
  with $W_i \in \mathbb{R}^{N^*_i \times N^*_i}$ and $b_i \in \mathbb{R}^{N^c_{i+1} \times N^c_{i+1}}$
  where $\ast$ is the matrix convolution operator. While $N^*_i$ stands for the dimension of the kernel $W_i$ of the corresponding convolution,
  we assume $N_i = N_i^c \cdot N_i^c$ and $N_{i+1} = N_{i+1}^c \cdot N_{i+1}^c$. The embedding into the multilayer perceptron is usually performed with the use of 
  $\operatorname{reshape}_N$ $( \mathbb{R}^{{N^2}} \to \mathbb{R}^{N \times N})$ and
  $\operatorname{flatten}_N$ $(\mathbb{R}^{N\times N} \to \mathbb{R}^{N^2})$  operators so that the dimensions of convolutional layer matches with the  dense layer.
  \end{remark}  
\begin{remark} 
In a scenario where a dense layer $j$ of MLP is replaced with a convolutional layer, equation \eqref{eq:layer step} must be modified as 
\begin{equation*}
  \begin{split}
    \|{z_{j}(y) - z_{j}(y^{f_i})}\|_F
    & = 
    \|{
\sigma(W^j \ast z_{j-1}(y))
-
\sigma(W^j \ast z_{j-1}(y^{f_i}))
    }\|_F
    \\
    & \leq  c_0
    \|{
W_j \ast \left(
z_{j-1}(y)
-
z_{j-1}(y^{f_i})
\right)
    }\|_F
    \leq
    c_0 \|{W_j}\|_F 
    \cdot 
    \|{z_{j-1}(y) - z_{j-1}(y^{f_i})}\|_F.
  \end{split}
\end{equation*}
Hence, for a neural network with an index set of dense layers $S_{d}$ and convolutional layers $S_c$ the result  \eqref{lemma:1:bound} holds with the modified constant
  \begin{equation*}
    c_W =  \prod_{j\in S_{d}} \|{W_j}\|_2 \prod_{j\in S_{c}} \|{W_j}\|_F
    \end{equation*}
    by taking into account, that 
  $\|\operatorname{reshape}(\cdot)\|_F = \|\cdot\|_2$ and $\|\operatorname{flatten}(\cdot)\|_2 = \|\cdot\|_F$.
\end{remark}
\section{Numerical experiments}
We consider the two-dimensional Poisson equation on the unit square $\Omega = (0,1)^2$ with homogeneous Dirichlet boundary conditions.
The training data is picked randomly from the set of source terms
\begin{multline}\label{Fset}
  {\cal F} :=
  \Big\{
  f(x,y) = \sum_{i=1}^4 \alpha_i\sin\big(\beta_i\pi(x+C_i)\big),\;
  C_1,C_2\in [0,1],\; C_3,C_4\in [0,\frac{1}{2}],\\
  \alpha_1=\alpha_2=\frac{1}{2},\; \alpha_3=\alpha_4=\frac{1}{10},\;
  \beta_1=\beta_2=2,\; \beta_3=\beta_4=4\Big\}
\end{multline}
together with the corresponding fine and coarse finite element solutions $u_H$ and $u_h$, respectively.
We employ a multilayer perceptron as described in Definition~\ref{def:MLP} with 4 hidden layers, each with 512 neurons and $\sigma(\cdot)=\tanh(\cdot)$ as an activation function. We train it using the Adam optimizer~\cite{kingma2017adam} and loss function $\mathcal{L}$ from \ref{loss}.

\begin{figure}
  \centering
\begin{minipage}[b]{0.4\textwidth}
    \begin{tikzpicture}[scale=0.8]
		  \begin{axis}[xlabel=$h$,
        ylabel=error,
        xmode=log,
        log basis x={2},
        ymode=log,
        log basis y={2},
        legend columns = 2,
        legend style={nodes={scale=0.9, transform shape},at={(0.5,1.1)},     anchor=south,/tikz/every even column/.append style={column sep=0.5cm}},
        cycle list/Paired,
        cycle multiindex* list={
                        Paired
                            \nextlist
                        my marks
                            \nextlist
                        [2 of]linestyles
                            \nextlist
                        very thick
                            \nextlist
                    }
            ]
        \addplot table [x=h, y=uff-uc (train), col sep=comma, solid] {data/error_over_refinement.csv};
        \addplot table [x=h, y=uff-uc (test), col sep=comma, solid] {data/error_over_refinement.csv};

        \addplot table [x=h, y=uff-uf (train), col sep=comma, solid] {data/error_over_refinement.csv};
        \addplot table [x=h, y=uff-uf (test), col sep=comma, solid] {data/error_over_refinement.csv};

        \addplot table [x=h, y=uff-un (train n_rhss 128), col sep=comma, solid] {data/error_over_refinement.csv};
        \addplot table [x=h, y=uff-un (test n_rhss 128), col sep=comma, solid] {data/error_over_refinement.csv};

        \addplot table [x=h, y=uff-un (train n_rhss 2048), col sep=comma, solid] {data/error_over_refinement.csv};
        \addplot table [x=h, y=uff-un (test n_rhss 2048), col sep=comma, solid] {data/error_over_refinement.csv};

        \addplot table [x=h, y=uff-un (train n_rhss 16384), col sep=comma, solid] {data/error_over_refinement.csv};
        \addplot table [x=h, y=uff-un (test n_rhss 16384), col sep=comma, solid] {data/error_over_refinement.csv};
        \legend{
          $\norm{u_H - u_{\text{ref}}}$ (train),
          $\norm{u_H - u_{\text{ref}}}$ (test),
          $\norm{u_h - u_{\text{ref}}}$ (train),
          $\norm{u_h - u_{\text{ref}}}$ (test),
          $\norm{u_\NN - u_{\text{ref}}} (\text{train,} n=2^7)$,
          $\norm{u_\NN - u_{\text{ref}}} (\text{test,} n=2^7)$,
          $\norm{u_\NN - u_{\text{ref}}} (\text{train,} n=2^{11})$,
          $\norm{u_\NN - u_{\text{ref}}} (\text{test,} n=2^{11})$,
          $\norm{u_\NN - u_{\text{ref}}} (\text{train,} n=2^{14})$,
          $\norm{u_\NN - u_{\text{ref}}} (\text{test,} n=2^{14})$,
        }
		  \end{axis}
		  \end{tikzpicture}
      \captionof{figure}{Error for different refinement levels}
      \label{fig:error_ref_level}
\end{minipage}
\begin{minipage}[b]{0.4\textwidth}
    \begin{tikzpicture}[scale=0.8]
		  \begin{axis}[xlabel=epoch, ylabel=loss, ymode=log]
        \addplot +[mark=none]table [x=epochs, y=train loss, col sep=comma, solid] {data/loss_example.csv};
        \addplot +[mark=none]table [x=epochs, y=test loss, col sep=comma, solid] {data/loss_example.csv};
        \legend{
          train,
          test
        }
		  \end{axis}
		  \end{tikzpicture}
  \captionof{figure}{Example of loss function during the training}
  \label{fig:loss_example}
\end{minipage}
\end{figure}

Figure~\ref{fig:error_ref_level} shows the mean error of the proposed method w.r.t. a reference one, which is one level finer that the target one. Here we consider the error on training and testing datasets of different sizes. We also consider different refinement levels, i.e. $h=H/2, H/4$ and $H/8$. The $x$-axis corresponds to the fine step size $h$ and the $y$-axis to the mean error. Here, the two  topmost lines (blue) show the error of the coarse solution, which is used as an input to the neural network. The two bottom-most lines (green) show the error of the fine solution, used for the computation of the loss. The rest of the lines depict the errors of the proposed method for training data of different size. Here we observe that given enough data, one is able to get arbitrarily close to the fine solutions used for training.

Figure~\ref{fig:loss_example} shows an example of how the loss function behaves during the training. Here we have trained a network for 400 epochs and have used learning rate decay with a factor of 0.5 every 100 epochs. Due to this one can observe significant drops in the value of loss function at 100, 200 and 300 epochs.

\begin{figure}[t]
  \centering
  \begin{tabular}{ccccc}
    \rotatebox{90}{\phantom{XXXXXXXXX}Coarse Solution}&
    \includegraphics[width=0.45\textwidth]{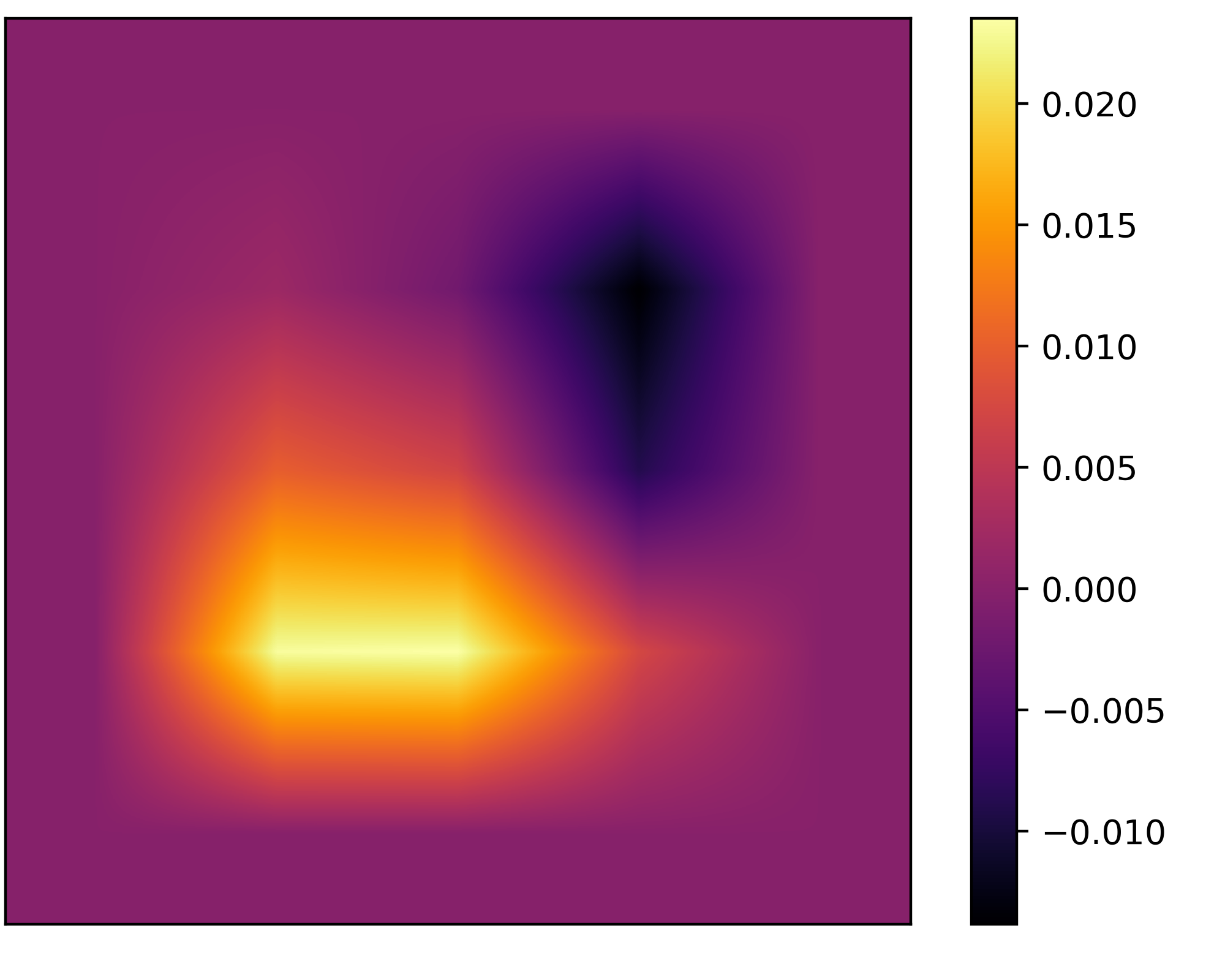}
    &&
    \includegraphics[width=0.45\textwidth]{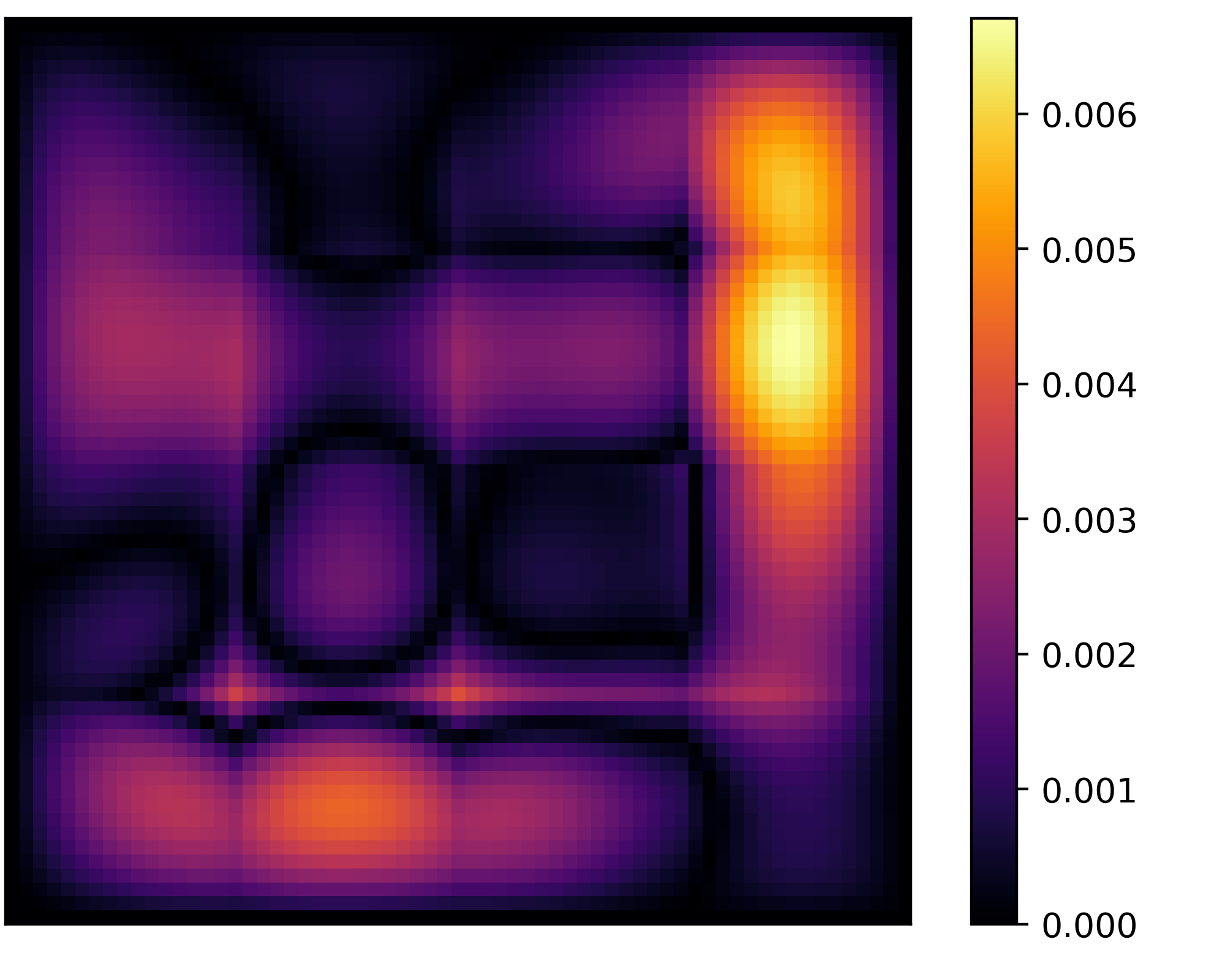}
    \\
    \rotatebox{90}{\phantom{XXXXXXXXXX}Fine Solution}&
    \includegraphics[width=0.45\textwidth]{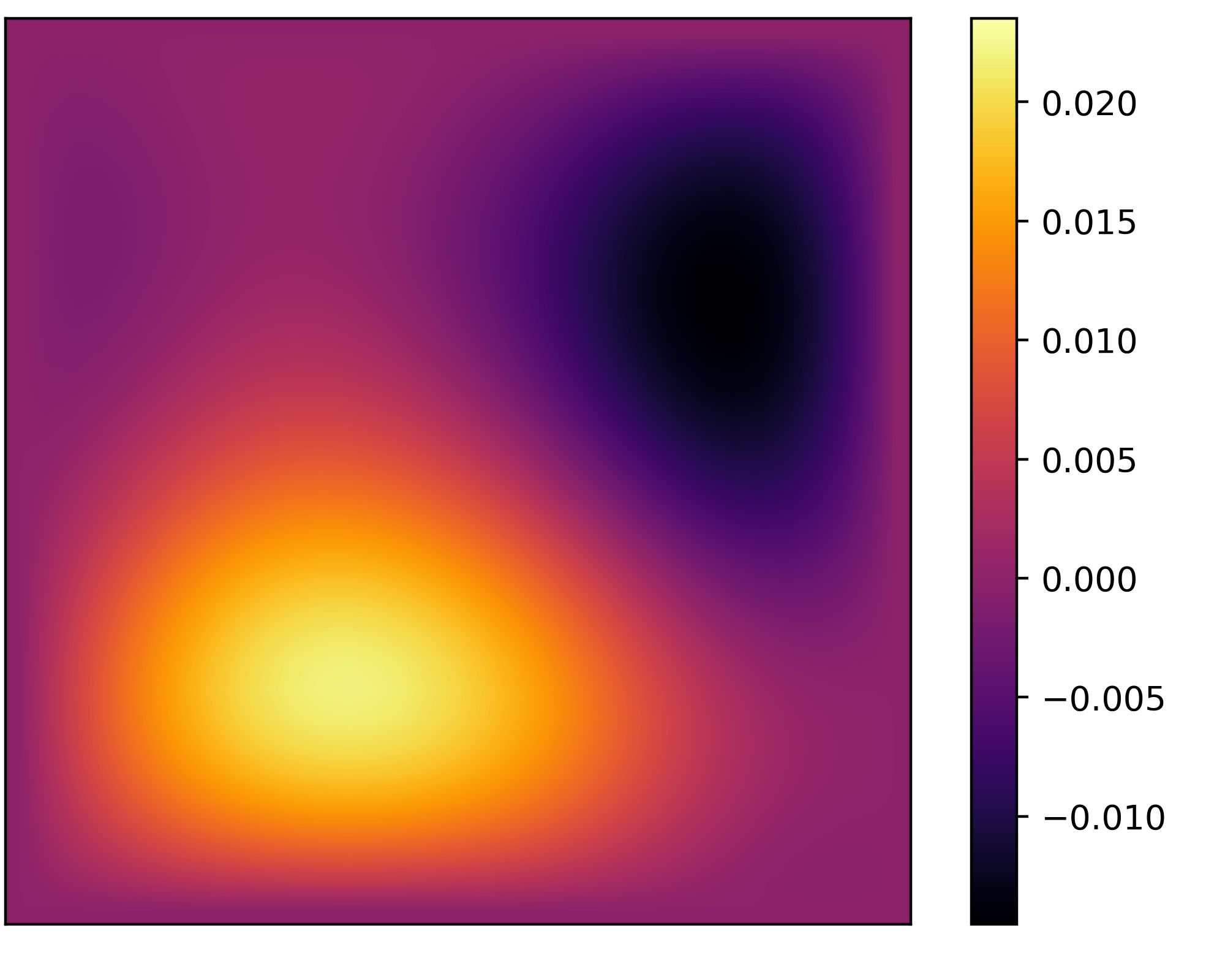}
    &&
    \includegraphics[width=0.45\textwidth]{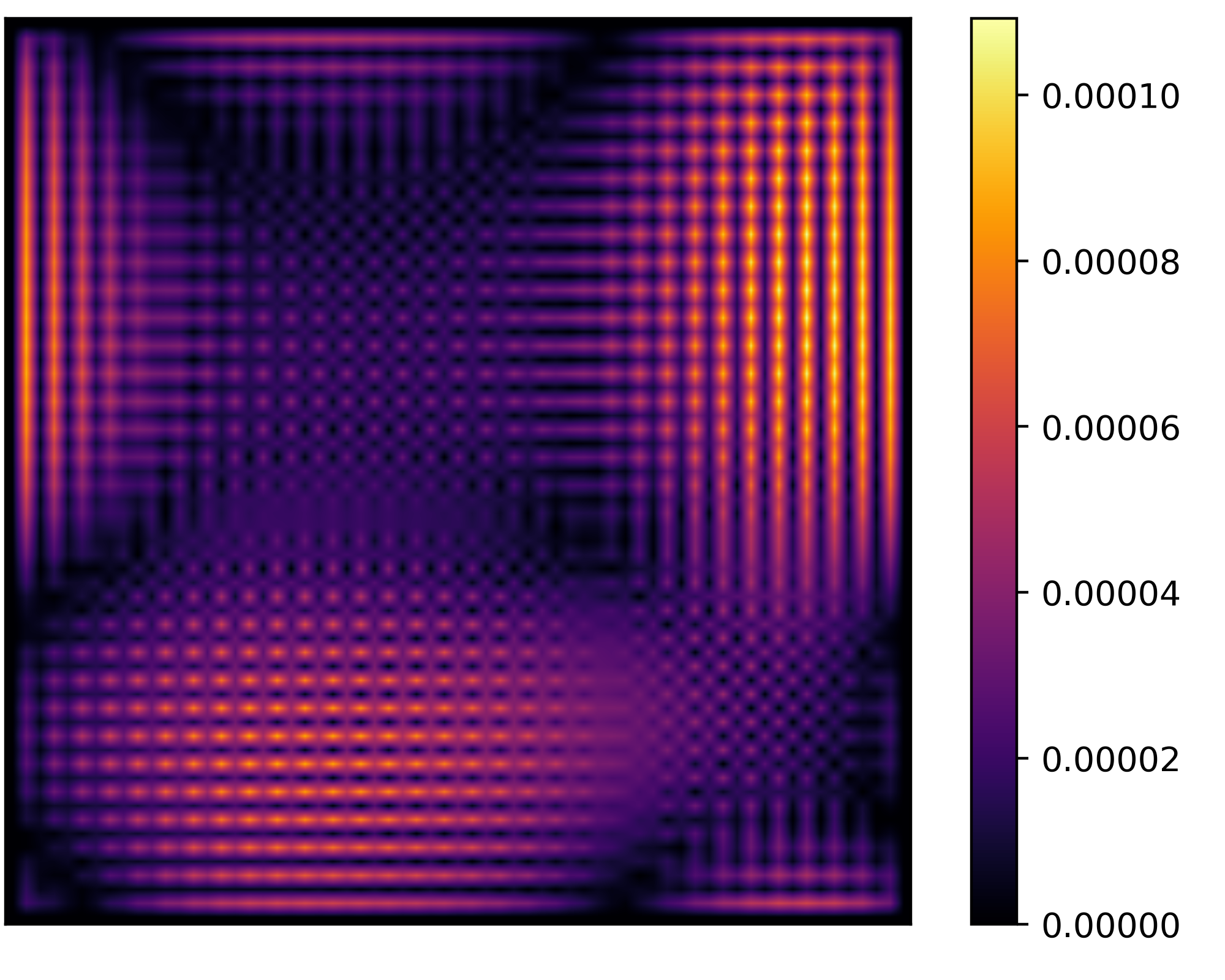}
    \\
    \rotatebox{90}{\phantom{XXXXXXXX}Hybrid NN Solution}&
    \includegraphics[width=0.45\textwidth]{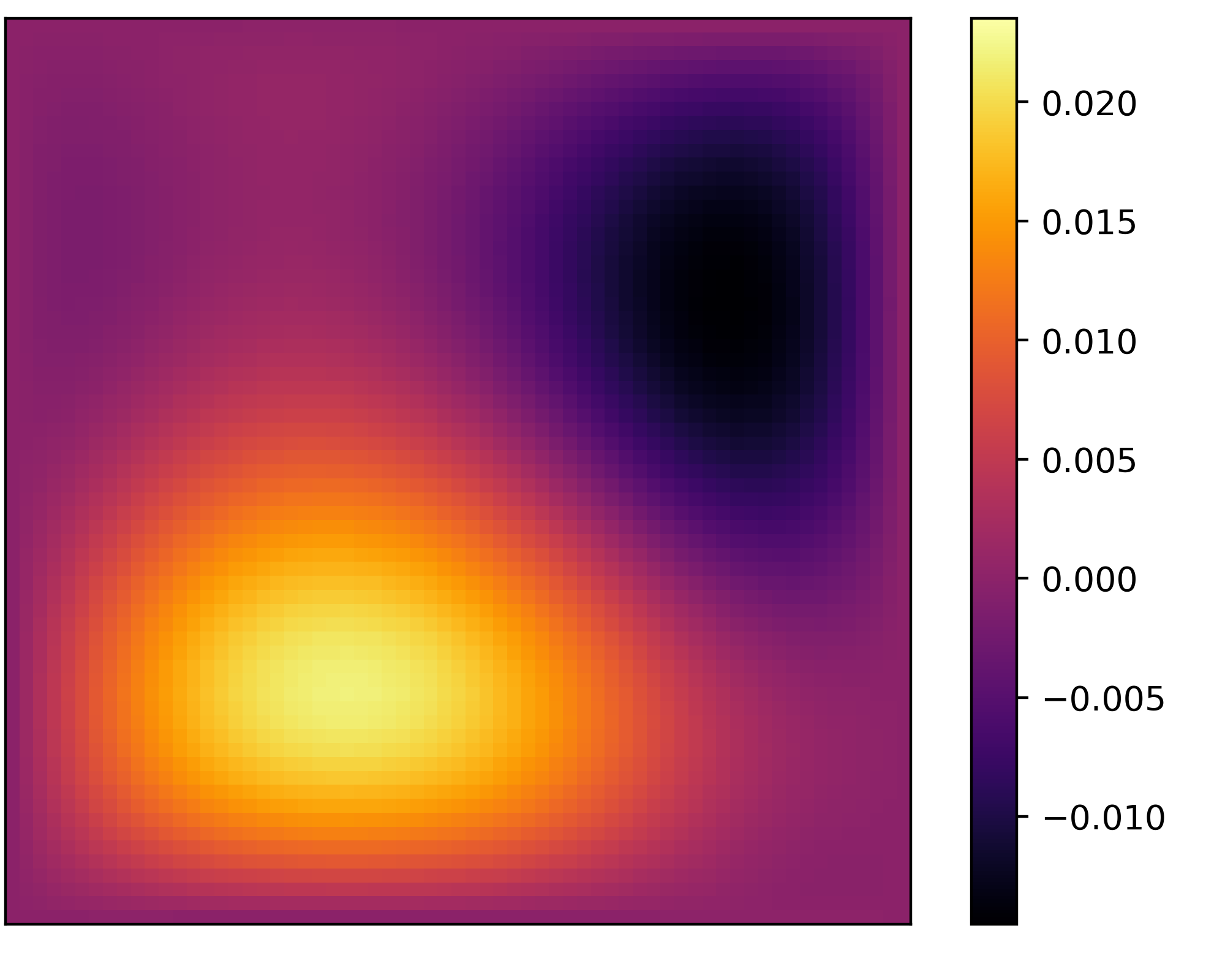}
    &&
    \includegraphics[width=0.45\textwidth]{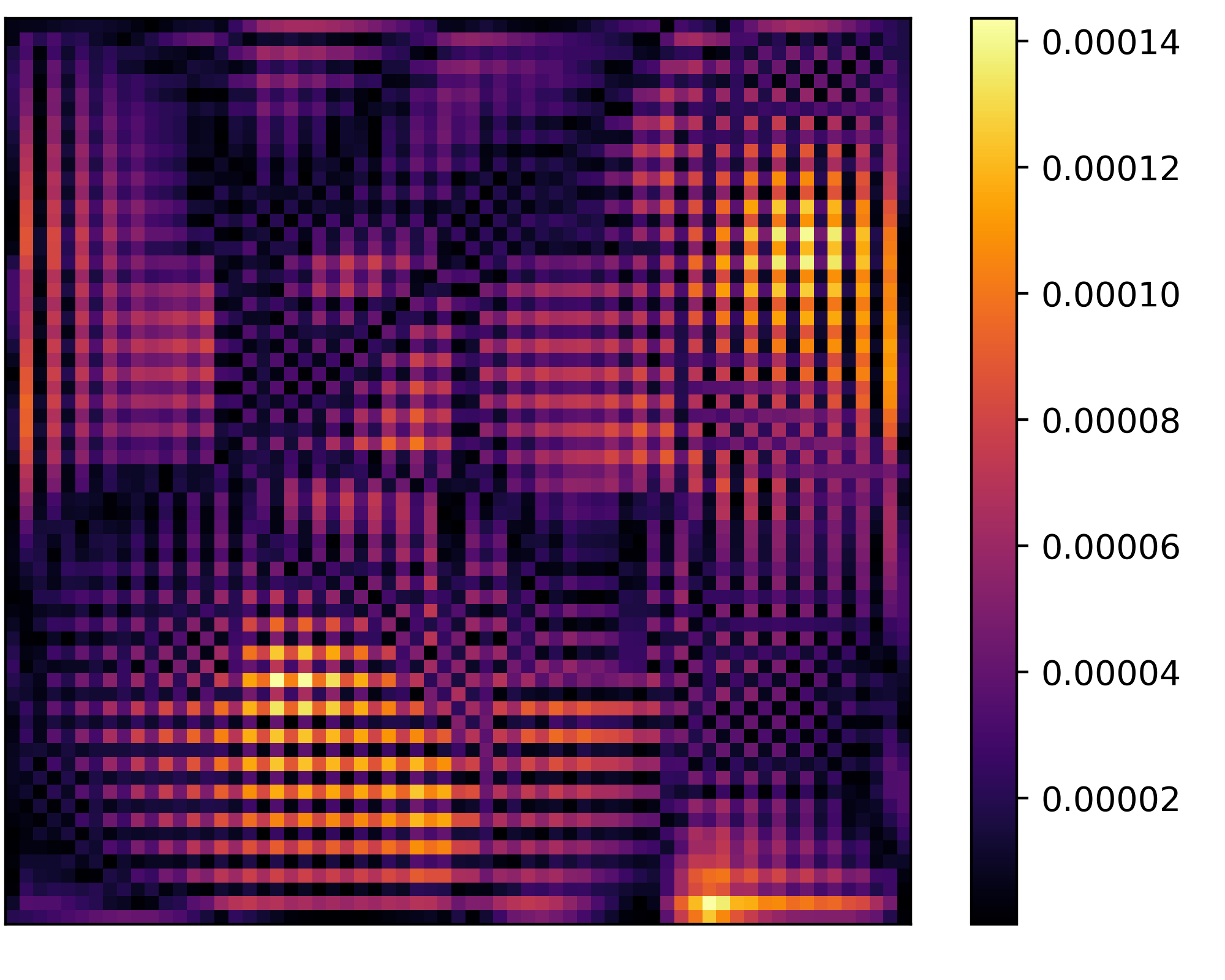}
  \end{tabular}
  \caption{Performance of the hybrid finite element - neural network approach. Top: coarse mesh solution $u_H$ and error. Middle: resolved fine mesh solution $u_h$ and error. Bottom: hybrid finite element - neural network solution $u_{\cal N}$ and error.}
  \label{fig:hybrid_solution_example}
\end{figure}

Figure~\ref{fig:hybrid_solution_example} shows an example of coarse, fine and network solution for a particular right hand side from the test data. Here we observe, that the quality of the network solution is significantly better than the quality of the original coarse solution.
\section*{Acknowledgements}
The authors acknowledge the support of the GRK 2297 MathCoRe, funded by the Deutsche Forschungsgemeinschaft, Grant Number 314838170.

\end{document}